\documentclass[12pt,twoside]{article}
\usepackage[a4paper,hmargin=1.5in,vmargin=1.5in]{geometry}
\usepackage{amsmath,amssymb,amsthm, hyperref}
\usepackage{graphicx,mathtools,bm,tabu}
\numberwithin{equation}{section}

\newtheorem{theorem}{Theorem}[section]
\newtheorem{thm}[theorem]{Theorem}

\newtheorem{lem}[theorem]{Lemma}

\theoremstyle{definition}

\newtheorem{defn}[theorem]{Definition}

\theoremstyle{remark}

\newtheorem{rem}[theorem]{Remark}

\newcommand{\Spec}{\operatorname{Spec}}
\renewcommand{\emptyset}{\varnothing}
\newcommand{\Diag}{\operatorname{diag}}
\renewcommand\footnotemark{}
\usepackage{fancyhdr}
\begin{document}

%
%
%
%
%
%
%
%
%

\title{A two-parameter eigenvalue problem for a class of block-operator matrices\footnote{Published in 
B\"{o}ttcher A., Potts D., Stollmann P., Wenzel D. (eds), \emph{The Diversity and Beauty of Applied Operator Theory}. Operator Theory: Advances and Applications, vol 268. Birkh\"{a}user, Cham, 2018; pp 367--380}\footnote{\url{https://doi.org/10.1007/978-3-319-75996-8_19}}
\footnote{MSC(2010) Primary 15A18; Secondary 47A25.}}

\author{Michael Levitin
\thanks{Authors' address: Department of Mathematics and Statistics,
University of Reading, Whiteknights, PO Box 220, Reading RG6 6AX, UK.}
\thanks{ML:\ M.Levitin@reading.ac.uk; HM\"{O}:\ H.M.Ozturk@pgr.reading.ac.uk.}
\and Hasen Mekki \"{O}zt\"{u}rk\
}

\date{}

\maketitle

\begin{abstract}
We consider a symmetric block operator spectral problem with two spectral parameters. Under some reasonable restrictions, we state localisation theorems for the pair-eigenvalues and discuss relations to a class of non-self-adjoint spectral problems.
\end{abstract}

\section{\label{sec:6}Introduction}

The Multiparameter Eigenvalue Problems (MEPs) are the generalisation
of the one-parameter standard eigenvalue problem $M\mathbf{x}=\lambda\mathbf{x}$
and the generalised one-parameter eigenvalue problem $M\mathbf{x}=\lambda V\mathbf{x}$.
MEPs can be written in the following abstract form: 
\begin{equation}
M\mathbf{x}=\sum_{i=1}^{k}\lambda_{i}V_{i}\mathbf{x},\label{eq:1.1}
\end{equation}
where $\lambda_{i}\in\mathbb{C}$, $i=1,2,\ldots,k$, are spectral parameters, and $M$ and $V_{i}$ are self-adjoint linear operators
in some Hilbert space $\mathcal{H}$. Then $\lambda=(\lambda_{1},\ldots,\lambda_{k})$
is called a \emph{multi-parametric eigenvalue} (or \emph{k-tuple}, or \emph{eigentuple}) if there exists an $\mathbf{x}\in\mathcal{H}\setminus  \{ 0  \} $, 
called an \emph{eigenvector}, such that \eqref{eq:1.1} holds.  

MEPs arise in numerous applications, in particular in mathematical physics
when the method of separation of variables is used to solve boundary
value problems for partial differential equations. In the 1960s, an abstract
algebraic setting for MEPs was introduced by Atkinson \cite{Atkinson1968,Atkinson1972}, see also \cite{AtkinsonMingarelli,Sleeman}  and references therein.

In this paper, we consider a special class of two-parameter eigenvalue problems in a block-operator setting. 
Let $\mathcal{H}_{1}$ and $\mathcal{H}_{2}$ be Hilbert spaces. Let in \eqref{eq:1.1},  with $k=2$, 
\[
\mathbf{x}=\begin{pmatrix}\mathbf{u}\\\mathbf{v}\end{pmatrix}\in  \mathcal{H}_{1}\oplus \mathcal{H}_{2},
\]
and 
\[
M=\begin{pmatrix}
A & C\\
C^{*} & B
\end{pmatrix},
\qquad V_{1}=\begin{pmatrix}
I & 0\\
0 & 0
\end{pmatrix},
\qquad V_{2}=\begin{pmatrix}
0 & 0\\
0 & I
\end{pmatrix},
\]
where $A$, $B$ are self-adjoint operators in the Hilbert spaces
$\mathcal{H}_{1}$, $\mathcal{H}_{2}$, respectively, and $C$ is
a linear operator from $\mathcal{H}_{2}$ to $\mathcal{H}_{1}$.  Hence the equation \eqref{eq:1.1} becomes 
\begin{equation}\label{eq:1.2}
M(\alpha,\beta)\begin{pmatrix}\mathbf{u}\\
\mathbf{v}
\end{pmatrix}=\begin{pmatrix}
A-\alpha & C\\
C^{*} & B-\beta
\end{pmatrix}
\begin{pmatrix}\mathbf{u}\\
\mathbf{v}
\end{pmatrix}=\mathbf{0}.
\end{equation}

In this paper, the operators $A$, $B$ and $C$ are assumed to be bounded, with further restrictions imposed starting from section \ref{sec:basics}. The case of unbounded operators will be considered elsewhere.

\begin{defn}\label{def:Alpha}
We call $(\alpha,\beta)\in\mathbb{C}^{2}$
a \emph{pair-eigenvalue} of $M$ if there exists a non-trivial solution
$\displaystyle\begin{pmatrix}\mathbf{u}\\
\mathbf{v}
\end{pmatrix}\in\mathcal{H}$ of \eqref{eq:1.2}. We denote by $\Spec_{p}(M)$
the set of all  pair-eigenvalues of $M$. If both $\alpha, \beta\in\mathbb{R}$, then we will call $(\alpha,\beta)$
a \emph{real pair-eigenvalue} of \eqref{eq:1.2}.
\end{defn}

The equation \eqref{eq:1.2} can be re-written as 
\begin{align}
(A-\alpha)\mathbf{u} & =-C\mathbf{v},\label{eq:1.3}\\
(B-\beta)\mathbf{v} & =-C^{*}\mathbf{u}.\label{eq:1.4}
\end{align}
If $\alpha\notin\Spec(A)$, then 
\eqref{eq:1.3} can be re-written as $\mathbf{u}=-(A-\alpha)^{-1}C\mathbf{v}$,
and substituting this into \eqref{eq:1.4} yields 
\begin{equation}
(B-C^{*}(A-\alpha)^{-1}C)\mathbf{v}=\beta\mathbf{v}.\label{eq:1.5}
\end{equation}
This also means that if $\alpha\notin\Spec(A)$
and $\beta(\alpha)$ is an eigenvalue of 
\[
B-C^{*}(A-\alpha)^{-1}C,
\]
then $(\alpha,\beta(\alpha))\in\Spec_p(M)$. 

\section{Basics and statements}\label{sec:basics}
\subsection{Restrictions and notation}

Suppose that $\mathcal{H}_{1}$ and $\mathcal{H}_{2}$ are finite
dimensional, and therefore we are dealing with matrices. In addition,
for simplicity, take $\mathcal{H}_{1}=\mathcal{H}_{2}=\mathcal{H}$ and $\dim\mathcal{H}=n$. Our main results (Theorem \ref{thm:FullChess} and its special case Theorem \ref{thm:SimpleChess}) are stated below. 

\begin{rem} Most of our results transfer rather seamlessly to the cases when $\mathcal{H}_{1}$ and $\mathcal{H}_{2}$  have either different finite dimensions, or are infinite dimensional, but we exclude these from this paper for clarity.
\end{rem}

The eigenvalues
of $A$ and $B$ will be denoted by 
\[
\alpha_{1} \leq  \ldots\leq\alpha_{n}, \qquad \beta_{1}  \leq  \ldots\leq\beta_{n},
\]
respectively, and their corresponding eigenvectors will be denoted
by $\boldsymbol{\varphi}_{j}$ and $\boldsymbol{\psi}_{j}$, $j=1,\dots,n$.

In stating most of our results, we restrict our attention to the case where $C$ has
rank one. Take $C=\kappa P$, where $\kappa\in\mathbb{R}$,
and $P$ is a projection onto a one-dimensional subspace $Z=\mathrm{Span}\{{\mathbf{z}}\}$, $\|\mathbf{z}\|=1$. In the basis $\{ \boldsymbol{\varphi}_j\}$, $P$ will have the matrix representation $\left (\langle \mathbf{z},\boldsymbol{\varphi}_k \rangle \langle \mathbf{z},\boldsymbol{\varphi}_j \rangle \right)_{k,j=1}^n$. 
The equation \eqref{eq:1.2} then becomes
\begin{equation}\label{eq:2.1}
\begin{pmatrix}
A-\alpha & \kappa P\\
\kappa P & B-\beta
\end{pmatrix}\begin{pmatrix}
\mathbf{u}\\
\mathbf{v}
\end{pmatrix}=\mathbf{0},
\end{equation}
and \eqref{eq:1.5} becomes 
\begin{equation}\label{eq:2.2}
(B-\kappa^2 P (A-\alpha)^{-1}P)\mathbf{v}=\beta\mathbf{v}.
\end{equation}
Thus, by \eqref{eq:2.2}, for every $\alpha\in\mathbb{C}\setminus\Spec(A)$,
there are $n$ complex values $\beta(\alpha)$, and the corresponding curves are continuous in $\alpha$.

Let $\Phi_{X,\lambda}$ denote the eigenspace of a self-adjoint operator $X$ corresponding
to an eigenvalue $\lambda$, simple or multiple. Further denote 
\begin{align*}
\Gamma_{X} & :=  \{ \lambda\in\Spec(X)\mid\exists\boldsymbol{\varphi}\in\Phi_{X,\lambda}:\boldsymbol{\varphi}\neq0\; \text{and}\;\langle \mathbf{z},\boldsymbol{\varphi}\rangle =0\} ,\\
\widetilde{\Gamma}_{X} & :=  \{ \lambda\in\Spec(X)\mid \langle \mathbf{z},\boldsymbol{\varphi}\rangle =0\ \forall\boldsymbol{\varphi}\in\Phi_{X,\lambda}\}.
\end{align*}
Note that $\widetilde{\Gamma}_{X}\subseteq \Gamma_{X}$. If $\lambda$ is a simple eigenvalue of $X$, then $\lambda\in\Gamma_{X}\iff\lambda\in\widetilde{\Gamma}_{X}$. Also,  $\Gamma_{X}$  contains all the multiple eigenvalues of $X$.

Let $Q:=I-P$ be the orthogonal projection onto $Z^\bot$. For a self-adjoint operator $X: \mathcal{H}\to\mathcal{H}$, denote  
\begin{align*}
X_{\Vert,\Vert} & =  \left. PX  \right |_Z\ : Z\to Z, \qquad X_{\bot,\Vert}  =  \left. PX  \right |_{Z^\bot}\ :Z^\bot\to Z,\\
X_{\Vert,\bot} & =  \left. QX  \right |_Z\ :Z\to Z^\bot, \quad X_{\bot,\bot} =  \left. QX  \right |_{Z^\bot}\ :Z^\bot\to Z^\bot.
\end{align*}
The eigenvalues
of $A_{\bot,\bot}$ and $B_{\bot,\bot}$ will be denoted by 
\begin{eqnarray*}
\widehat{\alpha}_{1}  \leq  \ldots\leq\widehat{\alpha}_{n-1}, \qquad \widehat{\beta}_{1}  \leq  \ldots\leq\widehat{\beta}_{n-1},
\end{eqnarray*}
respectively, and their corresponding eigenvectors will be denoted
by $\widehat{\boldsymbol{\varphi}}_{k}$ and $\widehat{\boldsymbol{\psi}}_{k}$, $k=1,\dots,n-1$.

\begin{rem}
By the variational principle, the eigenvalues of $A$ and $A_{\bot,\bot}$
interlace,
\begin{equation}
\alpha_{k}\leq\widehat{\alpha}_{k}\leq\alpha_{k+1},\label{eq:2.3}
\end{equation}
and similarly the eigenvalues of $B$ and $B_{\bot,\bot}$ interlace,
\[
\beta_{k}\leq\widehat{\beta}_{k}\leq\beta_{k+1}.
\]
\end{rem}

\subsection{Statement of the simple Chess Board Theorem}

Assume for the moment that $\Gamma_{A}=\Gamma_{B}=\emptyset$, which in particular implies that all the eigenvalues of $A$ and $B$ are simple.
Denote
\[
x_{0}:=-\infty, \quad\quad
x_{2n}:=\infty,\quad\quad
x_{2j-1}:=\alpha_{j},\quad \quad
x_{2k}:=\widehat{\alpha}_{k},
\]
and similarly for $\beta$, 
\[
y_{0}:=-\infty, \quad\quad
y_{2n}:=\infty,\quad\quad
y_{2j-1}:=\beta_{j},\quad \quad
y_{2k}:=\widehat{\beta}_{k},
\]
where $j=1,\dotsc,n$ and $k=1,\dotsc,n-1$. Then, the numbers 
$x_{0},\dotsc,x_{2n}$  
divide the $\alpha$-line into $2n$ intervals, finite or infinite, and similarly for $\beta$.
Combination of these lines
divides the $(\alpha,\beta)$-plane into rectangles, some
of them semi-infinite,
\[
R_{p,q}:=r_{p}\times r_{q},\quad r_{p}:=(x_{p-1},x_{p}),\quad
r_{q}:=(y_{q-1},y_{q}),\quad p, q=1,\dots,2n,
\]
see Figure \ref{fig:1}. 

\begin{figure}[!htb]
\begin{centering}
\includegraphics[width=0.7\textwidth]{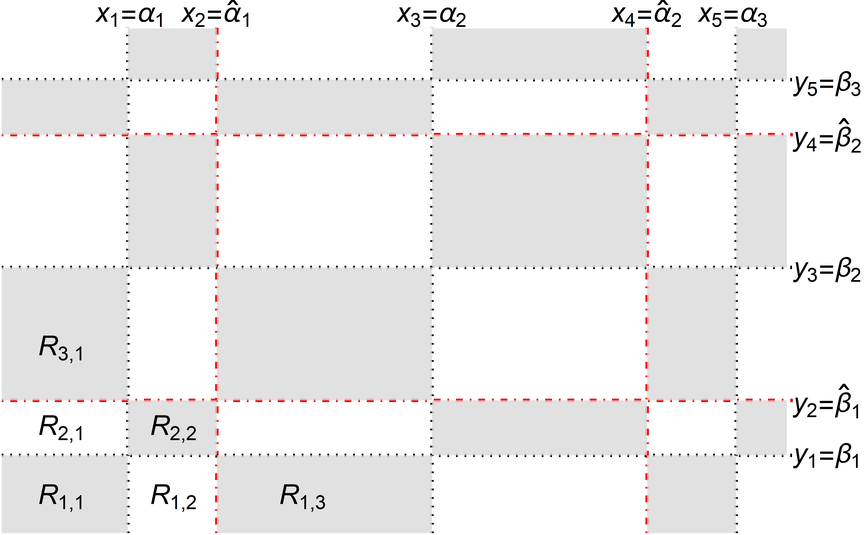}
\par\end{centering}
\caption{\label{fig:1} 
In the $(\alpha,\beta)$-plane, black
dotted and red dot-dashed lines in the vertical direction represent $\alpha=\alpha_{j}$
and $\alpha=\widehat{\alpha}_{k}$, respectively; in the horizontal direction
they represent $\beta=\beta_{j}$ and $\beta=\widehat{\beta}_{k}$, respectively. Here $n=3$, and the rectangles $R_{p,q}$ with even $p+q$ are shaded.}
\end{figure}

\begin{thm}[The Simple Chess Board Theorem]\label{thm:SimpleChess} 
Let $\Gamma_{A}=\Gamma_{B}=\emptyset$. Then all the real pair-eigenvalues
$(\alpha,\beta)$ of $M$ lie on a family of curves $(\alpha,\beta(\alpha))$ with the following properties:
\begin{enumerate}
\item[(a)] each curve may pass only through rectangles $R_{p,q}$ with $p+q$ even.
\item[(b)] each curve may cross from rectangle to rectangle only through the corner points $\left(x_{i},y_{j}\right)$ with $i+j$ odd;
\item[(c)] each curve $\beta(\alpha)$ is continuous in $\alpha$ except at eigenvalues of $A$; at each eigenvalue of $A$ exactly one curve blows up in the following sense:
$\beta(\alpha)  \rightarrow\pm\infty$ as $\alpha \rightarrow\alpha_{i}\pm0$, $\alpha_i\in\Spec(A)$;
\item[(d)] each curve $\beta(\alpha)$ is monotone decreasing in $\alpha$ on its domain of continuity; more precisely, we have
\begin{equation}\label{eq:mon}
\frac{d\beta}{d\alpha}=-\kappa^{2}\frac{\langle (A-\alpha)^{-2}\mathbf{z},\mathbf{z}\rangle (\langle (B-\beta)^{-1}\mathbf{z},\mathbf{z}\rangle)^{2}}{\langle (B-\beta)^{-2}\mathbf{z},\mathbf{z}\rangle }<0;
\end{equation}
\end{enumerate}
\end{thm}

\begin{rem} As $\alpha$ and $\beta$ are in fact interchangeable,  Theorem \ref{thm:SimpleChess} can be equivalently reformulated in terms of curves $(\alpha(\beta),\beta)$ with the only modification being that exactly one curve $\alpha(\beta)$ blows up at each eigenvalue of $B$ in the sense that $\alpha(\beta)  \rightarrow\pm\infty$ as $\beta \rightarrow\beta_{j}\pm0$, $\beta_j\in\Spec(A)$.
\end{rem}

\subsection{Statement of the full Chess Board Theorem}

In this section, we assume that either $ \Gamma_{A} \neq \emptyset $ or $ \Gamma_{B} \neq \emptyset $. Denote additionally, for $X: \mathcal{H} \rightarrow \mathcal{H}$, 
\begin{align*}
\Delta_X & :=  \{ \lambda\in\Spec(X)\mid \lambda \in \Spec(X)\cap \Spec(X_{\bot,\bot})\quad \text{and} \\
&\qquad \dim(\Phi_{X_{\bot,\bot},\lambda}) > \dim(\Phi_{X,\lambda}) \}.
\end{align*}
We will state formally an analogue of  Theorem \ref{thm:SimpleChess} below, but we start with summarising the principle changes: first, we exclude from the dividing mesh the points of  $\widetilde{\Gamma}_{A}\setminus \Delta_A$ and $\widetilde{\Gamma}_{B}\setminus \Delta_B$; and secondly, the real pair-spectrum of $M$ will, in addition to the curves, contain the lines $(\Gamma_{A}\times\mathbb{R})$ and 
($\mathbb{R}\times\Gamma_{B})$.

More precisely, let $x_i$, $i=1,\dots,s$, denote the points of 
\[
\left( (\Spec(A)\cup\Spec(A_{\bot,\bot}) )\setminus \widetilde{\Gamma}_{A}\right)\cup \Delta_A
\] 
enumerated in increasing order without account of multiplicities, and similarly $y_j$, $j=1,\dots,t$, denote the points of the analogue for $B$ enumerated in increasing order without account of multiplicities. Set additionally $x_0=y_0=-\infty$, $x_{s+1}=y_{t+1}=+\infty$, and 
\[
R_{p,q}=(x_{p-1},x_p)\times(y_{q-1},y_q),\quad p=1,\dots,s+1,\quad q=1,\dots,t+1.
\]

\begin{thm}[The Full Chess Board Theorem]\label{thm:FullChess} 
All the real pair-eigenvalues 
$(\alpha,\beta)$ of $M$ lie either on the straight lines $(\Gamma_{A}\times\mathbb{R}) \cup (\mathbb{R}\times\Gamma_{B})$ or on a family of curves $(\alpha,\beta(\alpha))$ with the following properties:
\begin{enumerate}
\item[(a)] each curve may pass only through rectangles $R_{p,q}$ with $p+q$ even;
\item[(b)] each curve may cross from rectangle to rectangle only through the corner points $\left(x_{i},y_{j}\right)$ with $i+j$ odd;
\item[(c)] each curve $\beta(\alpha)$ is continuous in $\alpha$ except at eigenvalues of $A$ not belonging to $\widetilde{\Gamma}_{A}$; at each such eigenvalue of $A$ exactly one curve blows up in the following sense:
$\beta(\alpha)  \rightarrow\pm\infty$ as $\alpha \rightarrow\alpha_{i}\pm0$, $\alpha_i\in\Spec(A)$;
\item[(d)] each curve $\beta(\alpha)$ is monotone decreasing in $\alpha$ on its domain of continuity with \eqref{eq:mon}.
\end{enumerate}
\end{thm}

\subsection{Limit cases}

In this section, we show that when $\kappa \rightarrow 0$, the components of the real pair-eigenvalues of \eqref{eq:2.1} approach the eigenvalues of $A$ and $B$, and when $\kappa \rightarrow \infty$, they approach the eigenvalues of $A_{\bot,\bot}$ and $B_{\bot,\bot}$. For brevity, we will work under the restrictions of the simple Chess Board Theorem.

\begin{thm}\label{thm:boards}
Suppose $\Gamma_A=\Gamma_B=\emptyset$. As $\kappa \rightarrow 0$, the real pair-eigenvalue spectrum of $M$ converges to 
$( \Spec(A) \times\mathbb{R} ) \cup (\mathbb{R} \times \Spec(B))$,
and as $\kappa \rightarrow \infty$, the real pair-eigenvalue spectrum of $M$ converges to
$( \Spec(A_{\bot,\bot}) \times\mathbb{R} ) \cup (\mathbb{R} \times \Spec(B_{\bot,\bot}))$.
\end{thm}

\section{Auxiliary results}
The statements in this section are for a single matrix, and mostly very elementary. We shall use them later in the proof of the Chess Board Theorem.
We shall frequently use the Fourier representation of the resolvent,
\begin{equation}\label{eq:resolvent}
(A-\alpha)^{-1}\mathbf{f}=\sum_{j}\frac{\langle \mathbf{f},\boldsymbol{\varphi}_{j}\rangle }{\alpha_{j}-\alpha}\boldsymbol{\varphi}_{j},\qquad\alpha\not\in\Spec(A).
\end{equation}
We also set 
\begin{equation}\label{eq:3.4}
R(\alpha):=\langle (A-\alpha)^{-1}\mathbf{z},\mathbf{z}\rangle =\sum_{j}\frac{|\langle \mathbf{z},\boldsymbol{\varphi}_{j}\rangle |^{2}}{\alpha_{j}-\alpha}.
\end{equation}

\begin{lem}\label{lem:3.1} 
Let $\alpha\notin\Spec(A)$. Then $R(\alpha) =0$ if and only if $\alpha\in\Spec(A_{\bot,\bot})$ and $(A-\alpha)^{-1}\mathbf{z}=c\widehat{\boldsymbol{\varphi}}$,
where $\widehat{\boldsymbol{\varphi}}$ is an eigenfunction of $A_{\bot,\bot}$ corresponding to $\alpha$ and $c\ne 0$.
\end{lem}
\begin{proof}

Set $\boldsymbol{\zeta}=(A-\alpha)^{-1}\mathbf{z}$. Then 
\begin{equation}\label{eq:3.1}
(A-\alpha)\boldsymbol{\zeta}=\mathbf{z} \Leftrightarrow  
\begin{pmatrix}
A_{\bot,\bot}-\alpha & A_{\Vert,\bot}\\
A_{\bot,\Vert}           & A_{\Vert,\Vert}-\alpha
\end{pmatrix}
\begin{pmatrix}
\mathbf{q}\\
\mathbf{p}
\end{pmatrix}=
\begin{pmatrix}
\mathbf{0}\\
\mathbf{z}
\end{pmatrix},
\end{equation}
where $\mathbf{q}=Q\boldsymbol{\zeta}$ and $\mathbf{p}=P\boldsymbol{\zeta}$.
Note that $\langle \boldsymbol{\zeta},\mathbf{z}\rangle =0$
iff $\mathbf{p}=P\boldsymbol{\zeta}=0$. Substituting this into \eqref{eq:3.1}
gives us 
\begin{equation}
\begin{cases}
(A_{\bot,\bot}-\alpha)\mathbf{q} & =\mathbf{0},\\
A_{\bot,\Vert}\mathbf{q} & =\mathbf{z}.
\end{cases}\label{eq:3.2}
\end{equation}
By the second equation, $\mathbf{q}$ is non-zero, and then by the first equation $\alpha\in\Spec(A_{\bot,\bot})$
and $\mathbf{q}=c\widehat{\boldsymbol{\varphi}}$, with $c\ne 0$.
Also, we have 
$
\mathbf{q}=Q\boldsymbol{\zeta}=(I-P) \boldsymbol{\zeta}= \boldsymbol{\zeta}
$,
and so $\boldsymbol{\zeta}=(A-\alpha)^{-1}\mathbf{z}=c\widehat{\boldsymbol{\varphi}}$.
\end{proof}

\begin{lem} \label{lem:CommonSpec}
$\Gamma_{A} = \emptyset$ if and only if $\Spec(A)\bigcap \Spec(A_{\bot,\bot}) = \emptyset$.
\end{lem}

\begin{proof}
If there exits an $\alpha \in\ \Gamma_{A}$, then there is an eigenfunction $ \boldsymbol{\varphi} \in \Phi_{A,\alpha} $ such that $\langle \mathbf{z},\boldsymbol{\varphi}\rangle =0$, and therefore $P\boldsymbol{\varphi}=0$ and so $Q\boldsymbol{\varphi}=\boldsymbol{\varphi}$. Thus
\[
A_{\bot,\bot}\boldsymbol{\varphi}=QA  \boldsymbol{\varphi} =\alpha Q \boldsymbol{\varphi}= \alpha\boldsymbol{\varphi},
\]
and so $\alpha\in \Spec(A)\bigcap \Spec(A_{\bot,\bot})$.

On the other hand, let  $\alpha\in \Spec(A)\bigcap \Spec(A_{\bot,\bot})$. Then 
\begin{align}
A\boldsymbol{\varphi}=\alpha \boldsymbol{\varphi} \quad \Rightarrow \quad 
\langle A\boldsymbol{\varphi}, \widehat{\boldsymbol{\varphi}} \rangle
=\alpha \langle \boldsymbol{\varphi}, \widehat{\boldsymbol{\varphi}} \rangle 
=\alpha \langle Q \boldsymbol{\varphi}, \widehat{\boldsymbol{\varphi}} \rangle.  \label{eq:3.2-2}
\end{align}
Also, since $  \widehat{\boldsymbol{\varphi}} \bot \boldsymbol{z} $, 
\begin{align}
A  \widehat{\boldsymbol{\varphi}} = A ( \widehat{\boldsymbol{\varphi}} + 0 \boldsymbol{z})
=A_{\bot,\bot}  \widehat{\boldsymbol{\varphi}} + A_{\bot,\parallel} \widehat{\boldsymbol{\varphi}}, \label{eq:3.2-3}
\end{align}
therefore 
\begin{align*}
\langle A\boldsymbol{\varphi}, \widehat{\boldsymbol{\varphi}} \rangle
 = \langle \boldsymbol{\varphi}, A \widehat{\boldsymbol{\varphi}} \rangle
& = \langle \boldsymbol{\varphi}, A_{\bot,\bot}  \widehat{\boldsymbol{\varphi}} \rangle
+\langle \boldsymbol{\varphi},  A_{\bot,\parallel} \widehat{\boldsymbol{\varphi}}  \rangle \\
&= \langle Q \boldsymbol{\varphi}, A_{\bot,\bot}  \widehat{\boldsymbol{\varphi}} \rangle
+\langle P \boldsymbol{\varphi},  A_{\bot,\parallel} \widehat{\boldsymbol{\varphi}}  \rangle \\
& = \alpha  \langle Q \boldsymbol{\varphi},   \widehat{\boldsymbol{\varphi}} \rangle
+\langle P \boldsymbol{\varphi},  A_{\bot,\parallel} \widehat{\boldsymbol{\varphi}}  \rangle 
\end{align*}
which implies by \eqref{eq:3.2-2} that $\langle P \boldsymbol{\varphi},  A_{\bot,\parallel} \widehat{\boldsymbol{\varphi}}  \rangle =0 $.
Now, if $  P \boldsymbol{\varphi} =0$, then $  \boldsymbol{\varphi} \bot \boldsymbol{z}$ so that $ \alpha \in \Gamma_{A} $.
If $A_{\bot,\parallel} \widehat{\boldsymbol{\varphi}} = 0$, then we have from \eqref{eq:3.2-3}  that 
$A \widehat{\boldsymbol{\varphi}} = \alpha \widehat{\boldsymbol{\varphi}}$, and therefore $ \widehat{\boldsymbol{\varphi}} $  is an eigenfunction of $A$ such that 
$ \boldsymbol{z} \bot \widehat{\boldsymbol{\varphi}}$, so that $ \alpha \in \Gamma_{A} $.
\end{proof}

%

\begin{lem}
\label{lem:3.4} If $\alpha \in \Spec(A) \setminus \widetilde{\Gamma}_{A}$, then $R(t)$ has a singularity at $t=\alpha$. The function $R(t)$
changes sign when $ t $ passes through an $\alpha_{j}$, $j=1,\ldots,n$, or
an $\widehat{\alpha}_{k}$,  $k=1,\ldots,n-1$. 

If $\alpha\in\widetilde{\Gamma}_{A}$, then $(A-\alpha)^{-1}\mathbf{z}$
exists, and $R(t)$ is continuous at $t=\alpha$. It changes sign at this $\alpha$ if and only if additionally $\alpha\in \Delta_A$.
\end{lem}

\begin{proof}
If $\alpha_j \in \Spec(A) \setminus \widetilde{\Gamma}_{A}$, then there exits at least one $\boldsymbol{\varphi}_j \in \Phi_{A,\alpha}$ such that  $\langle \mathbf{z},\mathrm{\boldsymbol{\varphi}}_j \rangle \neq0$, and it can be seen from \eqref{eq:3.4} that  $R(t)$
goes to $\pm\infty$ as $\alpha  \rightarrow\alpha_{j}\mp0$. 
Furthermore, since $R(t)$ has zeros at $\alpha=\widehat{\alpha}_{k}$ by Lemma \ref{lem:3.1}, and also is a continuous function except at the poles $\alpha=\alpha_{j}$,
it changes sign every time $t$ passes through $\widehat{\alpha}_{j}$
as well.

The second statement follows immediately from  \eqref{eq:3.4} and the fact that $\mathbf{z}\perp \Phi_{A,\alpha}$, and the last statement can be shown  by 
considering $\left.A\right|_{\Phi_{A,\alpha}^\bot}$ and repeating the above argument.
\end{proof}

\section{Proofs of the  main results}

We proceed to the proof of Theorem \ref{thm:FullChess}; Theorem \ref{thm:SimpleChess} follows from Theorem \ref{thm:FullChess} immediately  as a special case.

We first derive the characteristic equation of \eqref{eq:2.1}.
\begin{thm}
\label{thm:4.1}If $\alpha\notin\Spec(A)$ and $\beta\notin\Spec(B)$,
then the characteristic equation of \eqref{eq:2.1} for $\beta(\alpha)$
is 
\begin{equation}\label{eq:4.1}
\kappa^{2}\langle (A-\alpha)^{-1}\mathbf{z},\mathbf{z}\rangle \langle (B-\beta)^{-1}\mathbf{z},\mathbf{z}\rangle =1.
\end{equation}
\end{thm}

\begin{proof}
Re-writing the equation \eqref{eq:2.2} as 
\begin{equation}
 (B-\beta )\mathbf{v}=\kappa^{2}P (A-\alpha )^{-1}P\mathbf{v}\label{eq:4.3}
\end{equation}
and then using the information that $P$ is a projection, we obtain 
\begin{align*}
 (B-\beta )\mathbf{v} =\kappa^{2} \langle \mathbf{v},\mathbf{z} \rangle P (A-\alpha )^{-1}\mathbf{z} =\kappa^{2} \langle \mathbf{v},\mathbf{z} \rangle  \langle  (A-\alpha )^{-1}\mathbf{z},\mathbf{z} \rangle \mathbf{z}
\end{align*}
which implies
\[
\mathbf{v}=\kappa^{2} \langle \mathbf{v},\mathbf{z} \rangle  \langle  (A-\alpha )^{-1}\mathbf{z},\mathbf{z} \rangle  (B-\beta )^{-1}\mathbf{z}.
\]
Now since the term $\kappa^{2} \langle \mathbf{v},\mathbf{z} \rangle  \langle  (A-\alpha )^{-1}\mathbf{z},\mathbf{z} \rangle $
is a constant, we can fix it as 
\begin{equation}
\kappa^{2} \langle \mathbf{v},\mathbf{z} \rangle  \langle  (A-\alpha )^{-1}\mathbf{z},\mathbf{z} \rangle =1\label{eq:4.4}
\end{equation}
by setting 
$\mathbf{v}:= (B-\beta )^{-1}\mathbf{z}$. Substituting this into \eqref{eq:4.4}, we arrive at \eqref{eq:4.1}.
\end{proof}

The next lemma shows that $(\Gamma_{A}\times\mathbb{C})\cup (\mathbb{C}\times\Gamma_{B})\subset\Spec_{p} (M )$, strengthening in fact the claim of   
Theorem \ref{thm:FullChess}.

\begin{lem}\label{lem:StraightLines} If $\alpha\in\Gamma_{A}$, then 
$(\alpha, \beta)\in\Spec_{p} (M )$ for all $\beta\in\mathbb{C}$. Similarly if $\beta\in\Gamma_{B}$, then 
$(\alpha, \beta)\in\Spec_{p} (M )$ for all $\alpha\in\mathbb{C}$.
\end{lem}

\begin{proof} We prove the first of these statements, the second is similar. 
Let $\alpha\in\Gamma_A$, and let $\boldsymbol{\varphi}\in\Phi_{A,\alpha}$ such that $\langle \boldsymbol{\varphi}, \mathbf{z}\rangle=0$. An immediate check shows that
$
\begin{pmatrix}
\mathbf{u}\\
\mathbf{v}
\end{pmatrix}
=
\begin{pmatrix}
\boldsymbol{\varphi}\\
\mathbf{0}
\end{pmatrix}
$
is a pair-eigenvector of  \eqref{eq:2.1} for a pair-eigenvalue $(\alpha, \beta)$ with an arbitrary $\beta\in\mathbb{C}$.
\end{proof}

In Lemma \ref{lem:StraightLines} we show what happens when $\alpha\in\Gamma_{A}$ or $\beta\in\Gamma_{B}$; our next result shows which points $(\alpha,\beta)$ may lie in $\Spec_p(M)$ when  $\alpha$ is an eigenvalue of $A$ outside of $\Gamma_A$.

\begin{lem} \label{lem:intersection}
Let $\alpha\in\Spec(A)\setminus\Gamma_{A}$, and $\beta\not\in\Gamma_B$. Then $(\alpha,\beta)\in\Spec_{p}(M) $ if and only if $\beta=\widehat{\beta}\in\Spec(B_{\bot,\bot})$.
Similarly, if  $\beta\in\Spec(B)\setminus\Gamma_{B}$,  and $\alpha\not\in\Gamma_A$, then $(\alpha,\beta)\in\Spec_{p}(M)$ if and only if $\alpha=\widehat{\alpha}\in\Spec(A_{\bot,\bot})$.
\end{lem}

\begin{proof} Once more, we only prove the first statement.
Let $\alpha \in \Spec(A)\setminus\Gamma_A$. Let us re-write \eqref{eq:1.3}, \eqref{eq:1.4} as
\begin{align}
 (A-\alpha )\mathbf{u}&=-\kappa \langle \mathbf{v},\mathbf{z} \rangle \mathbf{z},\label{eq:4.10}\\
\kappa \langle \mathbf{u},\mathbf{z} \rangle \mathbf{z}+ (B-\beta )\mathbf{v} & =\mathbf{0}.\label{eq:4.11}
\end{align}
Multiplying \eqref{eq:4.10} by $\boldsymbol{\varphi}\in \Phi_{A,\alpha}$, we get
\[
\langle (A-\alpha )\mathbf{u},\boldsymbol{\varphi}\rangle=\langle\mathbf{u}, (A-\alpha )\boldsymbol{\varphi}\rangle=0=-\kappa\langle\mathbf{v},\mathbf{z}\rangle\,\langle\mathbf{z},\boldsymbol{\varphi}\rangle.
\]
Since $\alpha\not\in\Gamma_A$, we have $\langle\mathbf{z},\boldsymbol{\varphi}\rangle\neq0$, and so $\langle\mathbf{v},\mathbf{z}\rangle=0$ (and so $P\mathbf{v}=\mathbf{0}$), and by \eqref{eq:4.10}, $\mathbf{u}=a\boldsymbol{\varphi}$, where the constant $a$ may or may not be zero. 

Substituting now $\mathbf{u}=a\boldsymbol{\varphi}$ into \eqref{eq:4.11}, and applying the projections $Q$ and $P$ to the result, we obtain
\begin{align}
B_{\bot,\bot}\mathbf{v}&=\beta\mathbf{v},\label{eq:QB}\\
B_{\bot,\Vert}\mathbf{v}&=-\kappa a \langle\mathbf{z},\boldsymbol{\varphi}\rangle\mathbf{z}.\label{eq:PB}
\end{align}
 
If $\beta\not\in\Spec(B_{\bot,\bot})$, then by \eqref{eq:QB}, $\mathbf{v}=\mathbf{0}$, and thus $a=0$, and so $\mathbf{u}=\mathbf{0}$, and $(\alpha,\beta)\not\in\Spec_p(M)$, proving the ``only if'' part of the statement.

If $\beta=\widehat{\beta}\in\Spec(B_{\bot,\bot})$, and  $\widehat{\boldsymbol{\psi}}\in\Phi_{B_{\bot,\bot},\widehat{\beta}}$, we choose 
$\mathbf{v}=b\widehat{\boldsymbol{\psi}}$; we claim that we may choose constants $a,b$ such that $a^2+b^2\neq 0$ to satisfy \eqref{eq:PB}. 
After multiplying by $\mathbf{z}$, it becomes
\begin{equation}\label{eq:ab}
b \langle B_{\bot,\Vert}\widehat{\boldsymbol{\psi}},\mathbf{z}\rangle=-\kappa a \langle\mathbf{z},\boldsymbol{\varphi}\rangle.
\end{equation}
The scalar product on the right-hand side is non-zero by our assumption $\alpha\not\in \Gamma_A$. The scalar product on the left-hand side is non-zero since 
otherwise $\widehat{\beta}\in\Spec(B)$, and therefore $\beta\in\Gamma_B$ by Lemma  \ref{lem:CommonSpec}, again contradicting our assumptions. 
Thus we can always choose $a,b$ with $a^2+b^2\neq 0$ in order to satisfy \eqref{eq:ab}.
\end{proof}

We can now prove our main result.

\begin{proof}[Proof of the full Chess Board Theorem] The eigenvalues inside $(\Gamma_{A}\times\mathbb{R})\cup (\mathbb{R}\times\Gamma_B)$ have been already accounted for by Lemma \ref{lem:StraightLines}, so we will be working outside this set.

Recall the characteristic equation \eqref{eq:4.1}. 
Since it  needs to be satisfied, $ \langle  (A-\alpha )^{-1}\mathbf{z},\mathbf{z} \rangle $
and $ \langle  (B-\beta )^{-1}\mathbf{z},\mathbf{z} \rangle $ have to have the same sign  for real pair-eigenvalues.
It can be seen from \eqref{eq:3.4} that $ \langle  (A-\alpha )^{-1}\mathbf{z},\mathbf{z} \rangle $
is positive when $\alpha<\alpha_{1}$, and by 
Lemma \ref{lem:3.4}, it only changes sign every time when $\alpha$ passes through $x_p$, $p=1,\dots,s$. Similarly, $ \langle  (B-\beta )^{-1}\mathbf{z},\mathbf{z} \rangle $
is positive when $\beta<\beta_{1}$ and it only changes sign every time
when $\beta$ passes through $y_q$, $q=1,\dots,t$. Thus the only allowed regions for real $\alpha$ and $\beta$ are
when $ (\alpha,\beta )\in R_{p,q}$ with $p+q$ is even, proving, with account of Lemma \ref{lem:intersection},  the statements (a) and (b).

Statement (c) follows immediately from \eqref{eq:4.1} and Lemma \ref{lem:3.4}.

To prove (d), 
we differentiate the characteristic equation \eqref{eq:4.1} with respect to $\alpha$, arriving at
\[
\kappa^{2} \langle  (A-\alpha )^{-2}\mathbf{z},\mathbf{z} \rangle  \langle  (B-\beta )^{-1}\mathbf{z},\mathbf{z} \rangle +\kappa^{2} \langle  (A-\alpha )^{-1}\mathbf{z},\mathbf{z} \rangle  \langle  (B-\beta )^{-2}\mathbf{z},\mathbf{z} \rangle \frac{d\beta}{d\alpha}=0,
\]
so that 
\[
\frac{d\beta}{d\alpha}=-\frac{ \langle  (A-\alpha )^{-2}\mathbf{z},\mathbf{z} \rangle  \langle  (B-\beta )^{-1}\mathbf{z},\mathbf{z} \rangle }{ \langle  (A-\alpha )^{-1}\mathbf{z},\mathbf{z} \rangle  \langle  (B-\beta )^{-2}\mathbf{z},\mathbf{z} \rangle },
\]
and re-arranging with account of \eqref{eq:4.1}, we can re-write $\beta'$ as in \eqref{eq:mon}. 

To see that $\beta'<0$, we observe from  \eqref{eq:4.1} that $ \langle  (A-\alpha )^{-1}\mathbf{z},\mathbf{z} \rangle \neq0$
and $ \langle  (B-\beta )^{-1}\mathbf{z},\mathbf{z} \rangle \neq0$.
Also, 
\[
 \langle  (A-\alpha )^{-2}\mathbf{z},\mathbf{z} \rangle = \langle  (A-\alpha )^{-1}\mathbf{z}, (A-\alpha )^{-1}\mathbf{z} \rangle = \Vert  (A-\alpha )^{-1}\mathbf{z} \Vert,
\]
which is always positive by \eqref{eq:resolvent},  and similarly $ \langle  (B-\beta )^{-2}\mathbf{z},\mathbf{z} \rangle >0$, and therefore $ d\beta / d\alpha <0$.
\end{proof}

\begin{proof}[Proof of Theorem \ref{thm:boards}]
By the characteristic equation \eqref{eq:4.1}, we have, as $\kappa \rightarrow 0$, that either
$
\langle (A-\alpha_\kappa)^{-1}\mathbf{z},\mathbf{z}\rangle \rightarrow \infty
$
or
$
\langle (B-\beta_\kappa)^{-1}\mathbf{z},\mathbf{z}\rangle \rightarrow \infty
$,
and the first statement follows by Lemma \ref{lem:3.4} and standard perturbation techniques.
Similarly, if $\kappa \rightarrow \infty$, then either 
$
\langle (A-\alpha_\kappa)^{-1}\mathbf{z},\mathbf{z}\rangle \rightarrow 0
$
or
$
\langle (B-\beta_\kappa)^{-1}\mathbf{z},\mathbf{z}\rangle \rightarrow 0
$,
and the result follows from Lemma \ref{lem:3.1}.
\end{proof}

\section{Examples}

\subsection{Motivation and Example 1}

The main motivation of this paper comes from the particular non-self-adjoint problem which was considered in \cite{DaLe}, with corresponding
change of notations. Consider the $n\times n$ matrices 
\[
A_1=\begin{pmatrix}
0 & 1\\
1 & 0 & \ddots\\
 & \ddots & \ddots & 1\\
 &  & 1 & 0
\end{pmatrix},
\qquad P_1 = \begin{pmatrix}
0\\
 & \ddots\\
 &  & 0\\
 &  &  & 1
\end{pmatrix}.
\]
We set $A=B=A_1$ and $C=\kappa P_1$ (i.e. $\mathbf{z}=(0,\dots,0,1)^T$). The eigenvalues of $A_1$ are given by 
\[
\alpha_{j}=2\cos  \left( \frac{\pi j}{n+1} \right), \quad j=1,\ldots,n,
\]
and the eigenvalues of $(A_1)_{\bot,\bot}$ are given by the same formula with $n$ replaced by $n-1$.

In fact, \cite{DaLe} studied the spectrum of a non-self-adjoint problem
\begin{equation} \label{eq:nsa}
\begin{pmatrix}
A_1+\gamma & P_1 \\
-P_1 & -A_1-\gamma
\end{pmatrix}
\begin{pmatrix}
\mathbf{u} \\
\mathbf{v}
\end{pmatrix}
= \lambda 
\begin{pmatrix}
\mathbf{u} \\
\mathbf{v}
\end{pmatrix},
\end{equation}
where $\lambda$ is a spectral parameter and $\gamma \in \mathbb{R}$ is fixed; the problem \eqref{eq:nsa} relates to \eqref{eq:2.1} by setting $\kappa=1$ and
\begin{equation} \label{eq:Change}
\alpha= \lambda - \gamma; \qquad \beta= -\lambda-\gamma.
\end{equation}
We shall return to the comparison of the two problems and especially to non-real $\lambda$ in Section \ref{sec:nsa}.

\begin{figure}[!htb]
\includegraphics[width=0.49\textwidth]{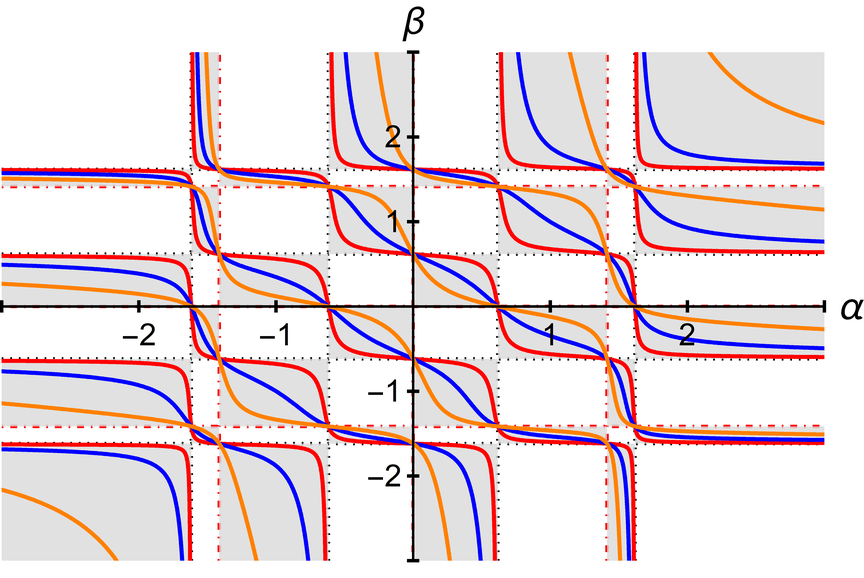}
\hfill
\includegraphics[width=0.49\textwidth]{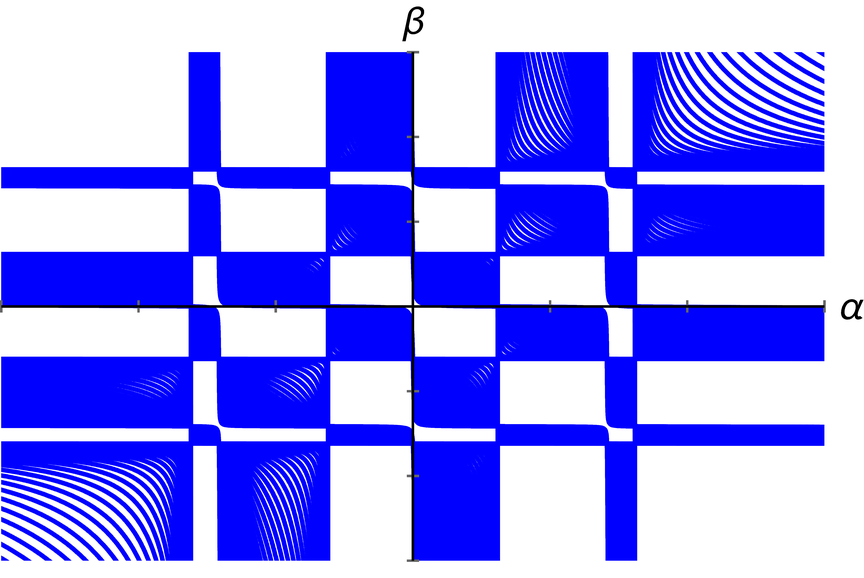}
\caption{\label{fig:forexample1} $A=B=A_1$ for $n=4$. Left:  $\Spec_{p}(M)$  with $\kappa=0.4$ (red curves), $\kappa=1$ (blue curves) and $\kappa=2$ (orange curves). 
Right: the superimposition of $\Spec_{p}(M)$  for the values of $\kappa$ from $0.001$ to $10$ with the step-size of $0.1$.}
\end{figure}

Note that $\Gamma_{A_1} = \emptyset $ and the general spectral picture in the $(\alpha,\beta)$-plane including the rectangular mesh can be seen in Figures \ref{fig:forexample1} and \ref{fig:GeneralComplex}. We see that the results of the simple Chess Board Theorem hold. 

\subsection{Example 2} \label{sec:ex2}
This example illustrates the case when $\Gamma = \emptyset $. We denote by $\Diag(\lambda_1, \dots, \lambda_n)$ a diagonal matrix composed of the entries $\lambda_1, \dots, \lambda_n$. Let $A_2= \Diag(-1,1)$ and $B_1= \Diag(1,3)$. Set $\mathbf{z}=\left( 1/\sqrt{5}, 2/\sqrt{5}\right)^T$. Then $\Gamma_{A_2}=\Gamma_{B_1} = \emptyset$. Also
$\Spec((A_2)_{\bot,\bot})=\left \{ -3/5 \right \}$,
$\Spec((B_1)_{\bot,\bot})=\left \{ 7/5 \right \}$.
The spectral picture can be seen in the left of Fig. \ref{fig:forexamples2and3}, and we see that the simple Chess Board Theorem (Theorem \ref{thm:SimpleChess}) holds.

\subsection{ Example 3}
This example illustrates two cases; the case when $\Gamma \neq \emptyset$ and $\widetilde{\Gamma}=\emptyset$, and also the case when $\Gamma =\widetilde{\Gamma} \neq \emptyset$. Consider $A_3=\Diag(-1,-1)$ and $B_1$. Set $\mathbf{z}= (1,0)^T$. Then $\Gamma_{A_3}=\{-1\} $, $\widetilde{\Gamma}_{A_3}=\emptyset$, and $\Gamma_{B_1} =\widetilde{\Gamma}_{B_1}= \{ 3 \}$. Also $\Spec((A_3)_{\bot,\bot}) = \{-1 \}$,
$\Spec((B_1)_{\bot,\bot})=\{ 3\}$.
The spectral picture is shown in the right of Fig. \ref{fig:forexamples2and3}. We see that $\Spec_p(M)$ has an additional vertical straight line at $\alpha =-1$, and there is also a blow-up at $\alpha =-1$. This line is included in the mesh since $\mathbf{z}$ is orthogonal to one eigenvector but $\mathbf{z} \not\perp \Phi_{A_3,-1}$. On the other hand, there is a horizontal straight line passing through $\beta=3$ which is not included in the mesh since $B_1$ has simple eigenvalues and $\mathbf{z} \perp \Phi_{B_1,3}$.

\begin{figure}[!htb]
\includegraphics[width=0.49\textwidth]{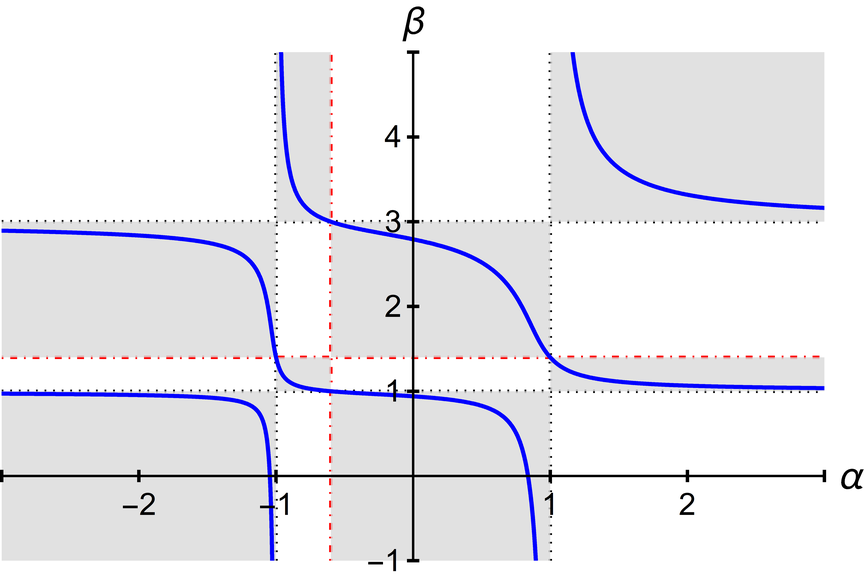}
\hfill
\includegraphics[width=0.49\textwidth]{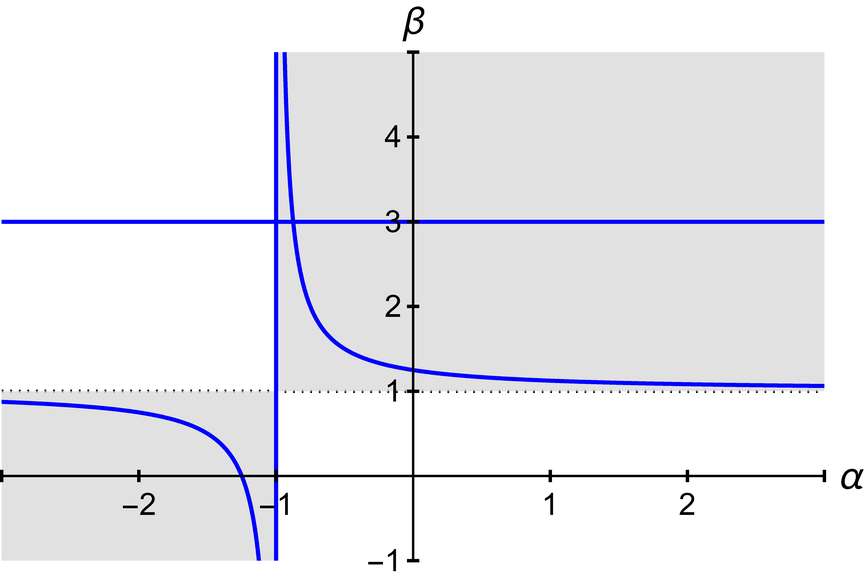}
\caption{\label{fig:forexamples2and3} $\Spec_{p}(M)$ for two cases. Left: with $A=A_2$, $B=B_1$ and $\kappa = 2/3$.
Right: with $A=A_3$, $B=B_1$ and $\kappa = 1/2$.}
\end{figure}

\subsection{Example 4}
This example illustrates the case when $\Gamma$, $\widetilde{\Gamma} \neq \emptyset$ and $\Gamma \neq \widetilde{\Gamma}$. Take
\begin{equation}\label{eq:B2}
A_4=\Diag(1,1,3,3),
\qquad
B_2=
\begin{pmatrix}
-2 & 1 & 0 & 0 \\
1 & -1 & 0 & 0 \\
0 & 0 & 2 & 1 \\
0 & 0 & 1 & 3
\end{pmatrix}.
\end{equation}
Set $ \mathbf{z} = (0,0,0,1)^T$. Then we have 
$\Spec(B_2)=\{ -\left(3\pm\sqrt{5}\right)/2,\left(5 \pm \sqrt{5}\right)/2 \}$ and $ \Gamma_{A_4} = \{ 1,3 \}$, $\widetilde{\Gamma}_{A_4} = \{1\}$ and $\Gamma_{B_2}=\widetilde{\Gamma}_{B_2} = \{ -(3 \pm \sqrt{5})/2\}$. We also obtain $\Spec((A_4)_{\bot,\bot})=\{1,3\}$, where the eigenvalue $1$ has multiplicity two, and $\Spec((B_2)_{\bot,\bot})= \{\left(-3 \pm \sqrt{5}\right)/2,2 \}$. The spectral picture is shown in the left of Fig. \ref{fig:forexamples4and5}. As expected, there are two additional vertical straight lines: at $\alpha=1$, where is no blow-up and the line is not included in the mesh since $\mathbf{z}\perp \Phi_{{A_4} ,1}$; and at $\alpha=3$, where is a blow-up and the line is included in the mesh since $\mathbf{z} \not\perp \Phi_{{A_4} ,3}$. On the other hand, there are two additional horizontal lines at $\beta=-(3\pm\sqrt{5})/2$ which are not included in the mesh as $\mathbf{z}$ is orthogonal to the corresponding eigenspaces.

\subsection{Example 5}
This example illustrates the case when $\Delta \neq \emptyset$. Consider $A_5=\Diag(1,2,2,3)$, and $B_2$. Set $\mathbf{z} = \left(1/\sqrt{2},0,0,1/\sqrt{2}\right)^T$. Then $ \Gamma_{A_5} =\widetilde{\Gamma}_{A_5} = \{2\}$ and $\Gamma_{B_2}=\widetilde{\Gamma}_{B_2} = \emptyset$.
Also $\Spec((A_5)_{\bot,\bot})=\{2\}$, where the eigenvalue $2$ has multiplicity three, and 
$\Spec((B_2)_{\bot,\bot})= \{ \left(1 \pm \sqrt{13}\right)/2, 1/2 \}$. The spectral picture is shown in the right side of Fig. \ref{fig:forexamples4and5}. Since $\mathbf{z} \perp \Phi_{A_5,2}$, there is no blow-up at $\alpha=2$. Nevertheless, this line is included in the mesh as $\Delta_{A_5}= \{ 2 \}$, that is, $\dim \left(\Phi_{(A_5)_{\bot,\bot},2}\right) >\dim \left(\Phi_{A_5,2}\right)$.

\begin{figure}[!htb]
\includegraphics[width=0.49\textwidth]{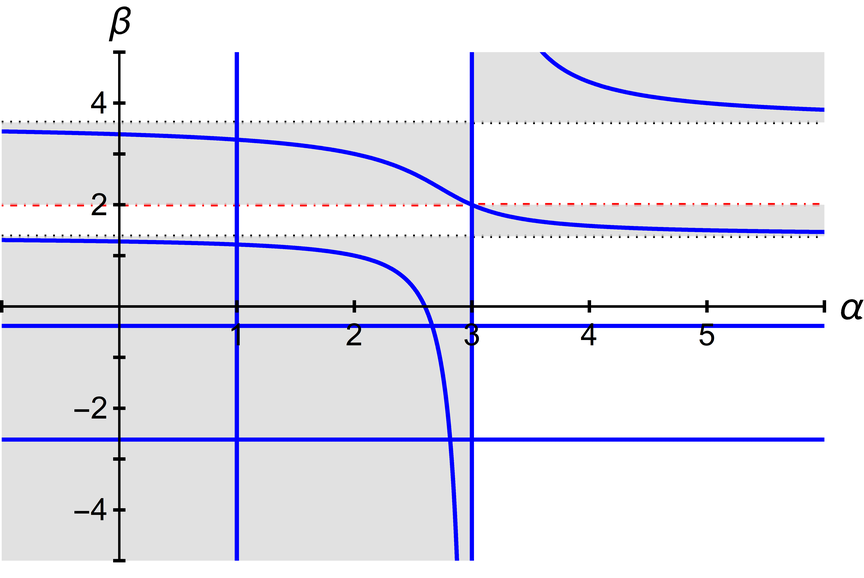}
\hfill
\includegraphics[width=0.49\textwidth]{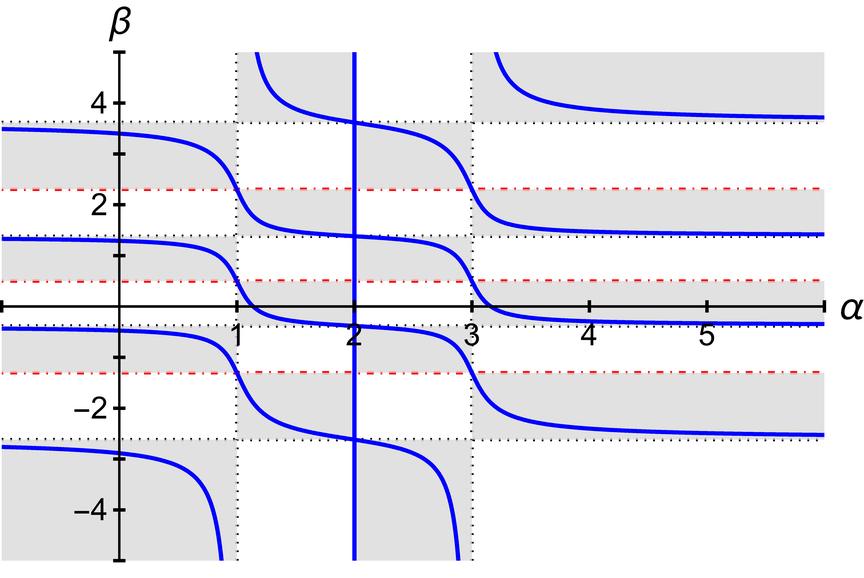}
\caption{\label{fig:forexamples4and5} $\Spec_{p}(M)$ with $\kappa = 1$ for two cases. Left: with $A=A_4$, $B=B_2$.
Right: with $A=A_5$, $B=B_2$.}
\end{figure}

\section{Relation to a non-self-adjoint problem} \label{sec:nsa}
We now return to the example studied in \cite{DaLe}. Generally speaking, there are $n$ complex $\beta (\alpha) \in \mathbb{C}$  for every $\alpha\in \Spec(A) \setminus \mathbb{C}$. We therefore limit our attention to pair-eigenvalues subject to the additional restriction
\begin{equation}\label{eq:Restriction}
\mathrm{Im}(\alpha+\beta)=0,
\end{equation}
which is equivalent to introducing the additional restriction $\gamma \in \mathbb{R}$, see \eqref{eq:Change}.

A general spectral picture of this non-self-adjoint problem in the $(\alpha,\beta)$-plane is illustrated in Figure \ref{fig:GeneralComplex}.
Red curves depict the real parts of non-real pair-eigenvalues $\mathrm{Re} \ \beta(\mathrm{Re} \ \alpha)$
such that \eqref{eq:Restriction} holds, which keeps all $(\alpha,\beta)\in\mathbb{R}^{2}$ in the
picture (shown in blue) and  also some non-real pair-eigenvalues. It is easily verified that the spectra are symmetric with respect to $(\alpha,\beta) \leftrightarrow (\beta,\alpha)$ and $(\alpha,\beta) \leftrightarrow (-\alpha, -\beta)$. 

\begin{figure} [!htb]
\begin{centering}
\includegraphics[width=0.96\textwidth]{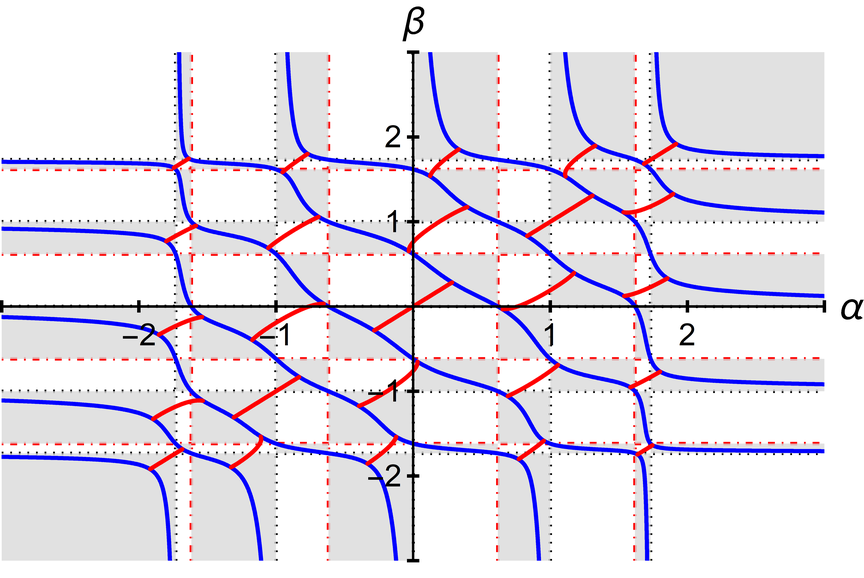}
\par\end{centering}
\caption{\label{fig:GeneralComplex}
 $\Spec_{p}(M)$ with $A=B=A_1$ and $C=P_1$ for $n=5$.}
\end{figure}

The real and non-real eigenvalue curves $\lambda(\gamma)$ may collide, with  two possible types of collisions: those when two real eigenvalues collide and produce a complex conjugate pair, called Type-A, and  those when a pair of complex
conjugate eigenvalues collide and become real, called Type-B,  see Figure \ref{fig:collision} for equivalents in $(\alpha,\beta)$-plane. 

\begin{figure}[!htb]
\includegraphics[width=0.24\textwidth]{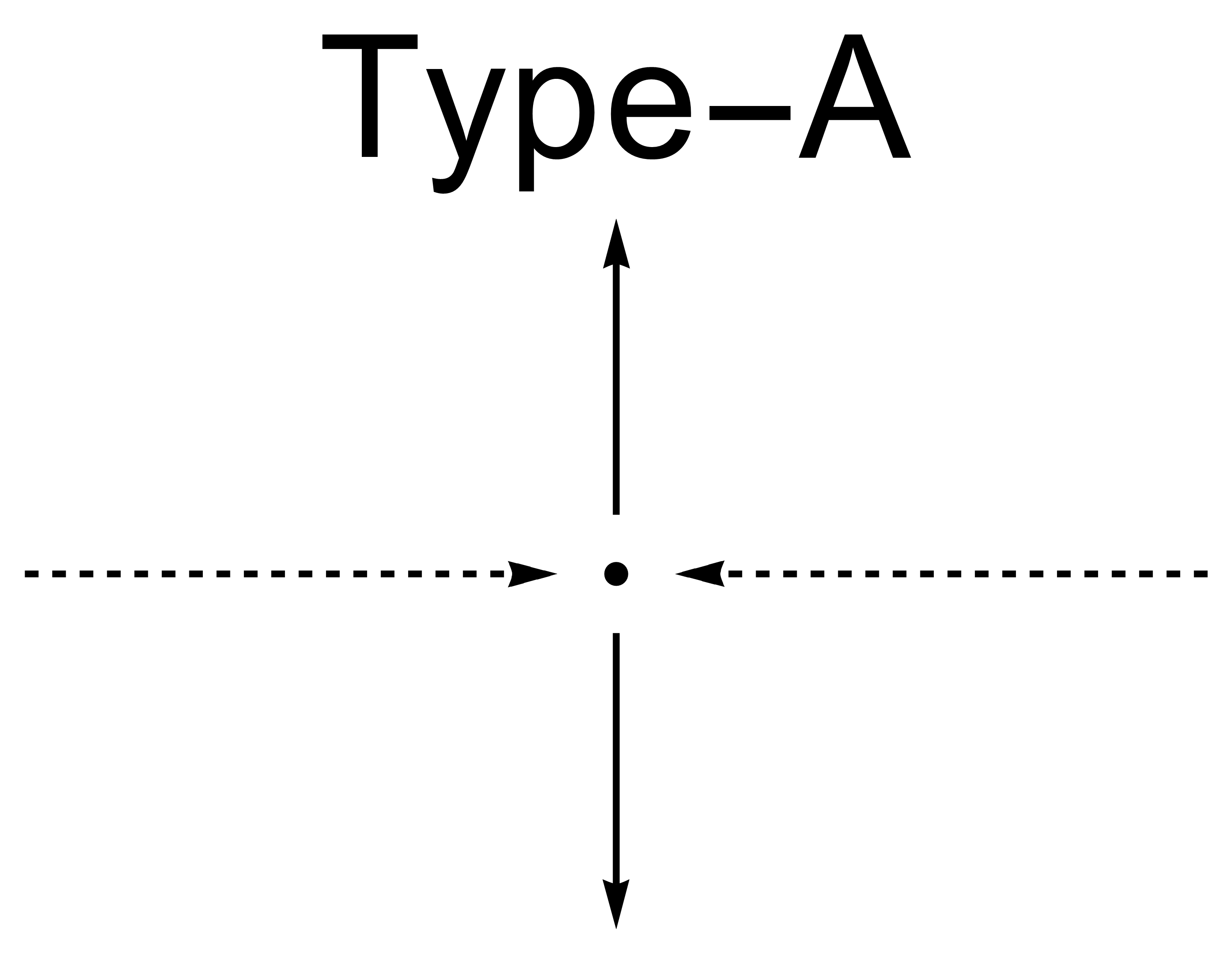}
\hfill
\includegraphics[width=0.24\textwidth]{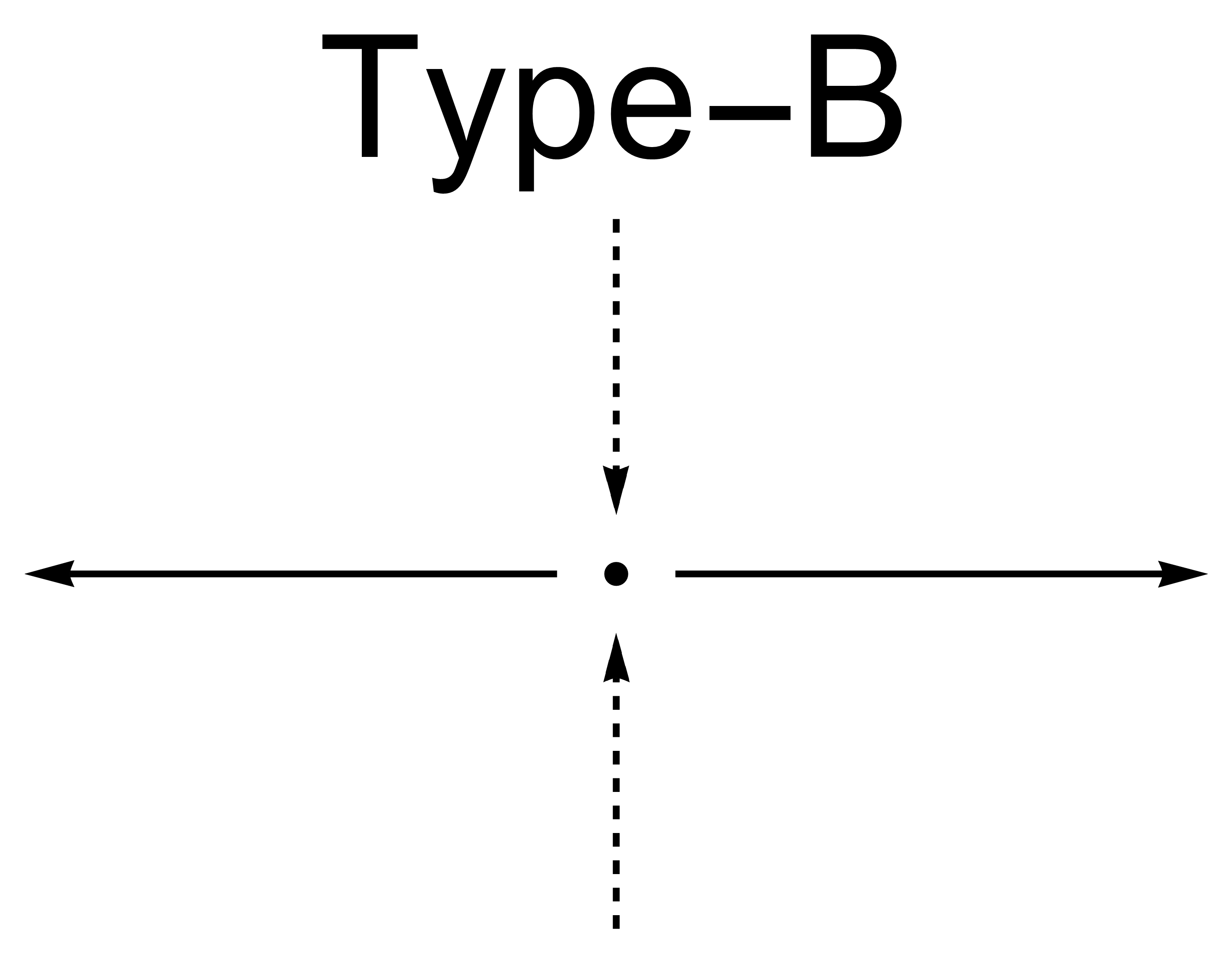}
\hfill
\includegraphics[width=0.24\textwidth]{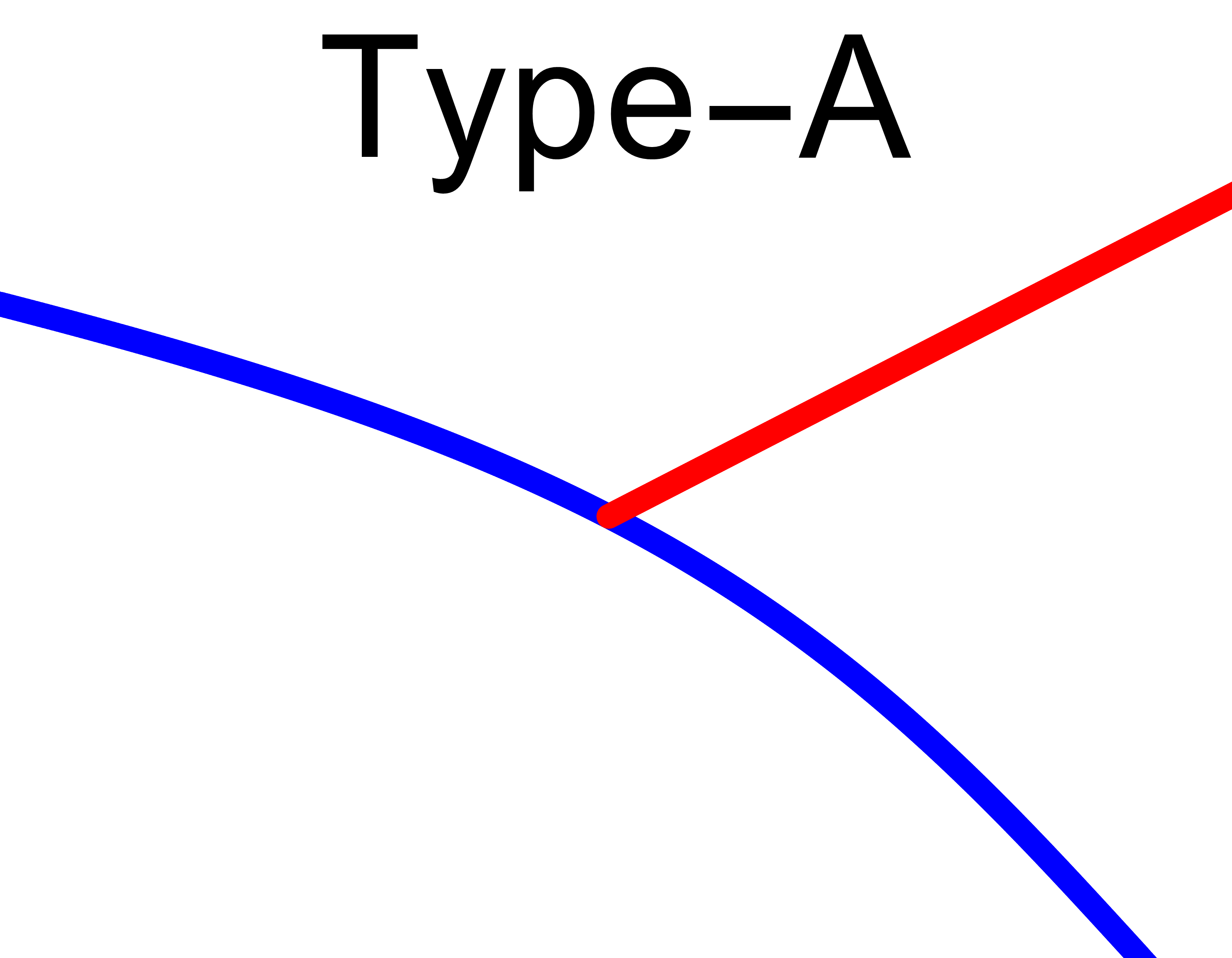}
\hfill
\includegraphics[width=0.24\textwidth]{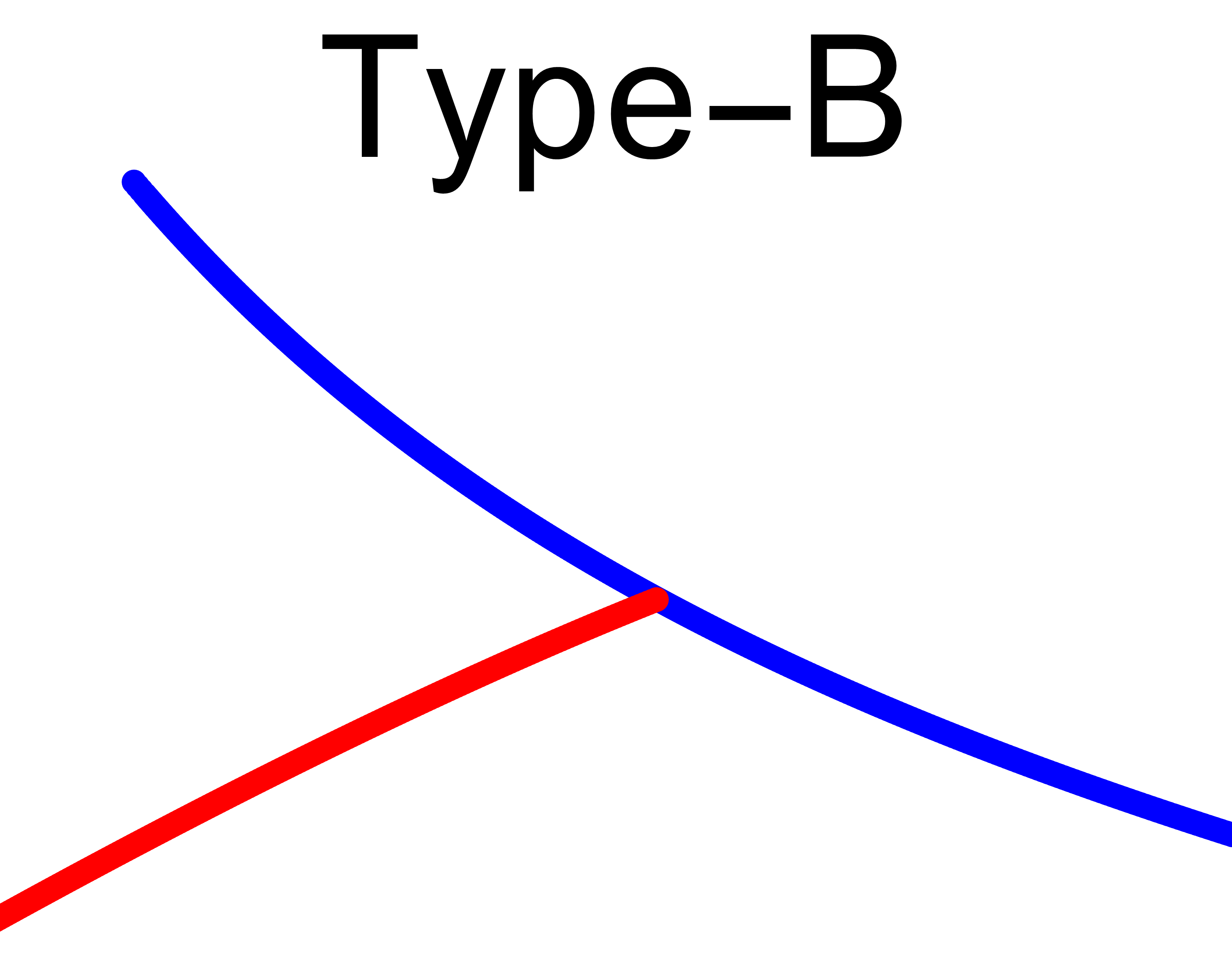}
\caption{\label{fig:collision} Left: the collisions in the $(\mathrm{Re}(\lambda),\mathrm{Im}(\lambda))$-plane.
Right: the collisions in the $(\alpha,\beta)$-plane.}
\end{figure}

\begin{lem}
The collisions happen at the points where $\frac{d\beta}{d\alpha}=-1$.
\end{lem}

\begin{proof}
Consider 
\begin{align*}
\begin{pmatrix}
A-\lambda+\gamma  & C \\
C^*                             & B+\lambda +\gamma
\end{pmatrix}
\begin{pmatrix}
\mathbf{u} \\
\mathbf{v}
\end{pmatrix}
= \mathbf{0}.
\end{align*}
Considering the curves $\gamma(\lambda)$ instead of $\lambda(\gamma)$ and differentiating with respect to $\lambda$ we arrive at 
\begin{equation}\label{eq:gammadiff}
\begin{pmatrix}
-1+\gamma'  &   \\
                     & 1+\gamma'
\end{pmatrix}
\begin{pmatrix}
\mathbf{u} \\
\mathbf{v}
\end{pmatrix}
= -\begin{pmatrix}
A-\lambda+\gamma  & C \\
C^*                            & B+\lambda +\gamma
\end{pmatrix}
\begin{pmatrix}
\mathbf{u'} \\
\mathbf{v'}
\end{pmatrix},
\end{equation}
which is solvable if and only if the right hand side of \eqref{eq:gammadiff} is perpendicular to $\begin{pmatrix} \mathbf{u} \\  \mathbf{v}  \end{pmatrix}$. Therefore multiplying by $\begin{pmatrix}\mathbf{u}\\ \mathbf{v} \end{pmatrix}$ we obtain
\begin{align*}
(-1+\gamma') \Vert \mathbf{u} \Vert ^2+(1+\gamma') \Vert \mathbf{v} \Vert^2 = 0 
\quad
 &\Leftrightarrow 
\quad
\gamma' (\Vert \mathbf{u} \Vert ^2+ \Vert \mathbf{v} \Vert ^2)= \Vert \mathbf{u} \Vert ^2-\Vert \mathbf{v} \Vert ^2 \\
& \Leftrightarrow
\quad
\gamma'=0,
\end{align*}
and since at critical points $\frac{d \gamma}{d\lambda}=0$, using \eqref{eq:Change} we obtain 
\begin{align*}
\frac{d \gamma}{d\lambda}=
\frac{\frac{d \gamma}{d\alpha}}{\frac{d\lambda}{d\alpha}}=\frac{-\frac{1}{2}-\frac{1}{2}\beta^{'}}{\frac{1}{2}-\frac{1}{2}\beta^{'}}=0\quad\Leftrightarrow\quad\beta^{'}=-1. 
\tag*{\qedhere} 
\end{align*}
\end{proof}


\subsection*{Acknowledgments}

We are grateful to E. Brian Davies for useful suggestions. 
The second author acknowledges the financial support by the Ministry of National Education of the Republic of Turkey.

\end{document}